\documentclass[leqno,12pt]{article} 
\usepackage{bbm}
\usepackage{mathrsfs}
\usepackage{hyperref}
\usepackage{eulervm} \usepackage[T1]{fontenc}
\usepackage{amsmath, amssymb}
\usepackage{amsthm} 
\theoremstyle{plain} 
\newtheorem{theorem}{\indent\sc Theorem}[section]
\newtheorem{lemma}[theorem]{\indent\sc Lemma}

\newtheorem{proposition}[theorem]{\indent\sc Proposition}

\theoremstyle{definition} 

\newtheorem{remark}[theorem]{\indent\sc Remark}

\title{Scaling limit of the random walk on a Galton-Watson tree with regular varying offspring distribution} 
\author{
\textsc{Dongjian Qian, Yang Xiao} 
}
\date{} 
\def\d{\mathrm{d}}
\def\TT{\mathbb{T}}

\begin{document}

\maketitle
\begin{abstract}
    We consider a random walk on a Galton-Watson tree whose offspring distribution has a regular varying tail of index $\kappa\in (1,2]$. We prove the convergence of the renormalised height function of the
    walk towards the continuous-time height process of a spectrally positive strictly stable
    Lévy process, jointly with the convergence of the renormalised trace of the walk towards
the continuum tree coded by the latter continuous-time height process.
\end{abstract}  

\section{Introduction}
 Let $\mathbb{T}$ be a Galton-Watson tree with root $\rho$. For every vertex $u$ of $\mathbb{T}$, denote by $\nu(u)$ the number of its children, which are i.i.d random variables with the generic law $\nu$. Throughout the article, we assume that the tail distribution of $\nu$ is regularly varying with index $\kappa\in (1,2]$ i.e.
\begin{align*}
    \mathbb P(\nu>x){\sim} x^{-\kappa}l(x),
\end{align*}
where $l(x)$ is a slowly varying function and we say $f(x)\sim g(x)$ if 
\begin{align*}
    \lim\limits_{x\rightarrow \infty} \frac{f(x)}{g(x)}=1.
\end{align*}
We may and will assume that $l(x)$ is a positive constant when $x$ is small. We see that $\nu$ has a finite moment of order $\alpha$ for $\alpha\in (0,\kappa)$. In particular, let $m:=E[\nu]<\infty$ be the expectation of reproduction. We assume $m>1$ so that $\TT$ has a positive probability to be infinite by knowledge of branching processes. 

For a vertex $u$ in the tree, we will denote by $|u|$ its generation, by $u^*$ its parent, and by $c(u)$ the set of its children. We also set $u_0,u_1,\dots u_{|u|-1}$ as the ancestors of $u$ at generations $0,1,\dots |u|-1$. For $u,v\in \mathbb{T}$, we write $u\leq v$ when $u$ is an ancestor of $v$ and $u<v$ when $u\le v$ and $u\neq v$. Let $\TT_u:=\{ v\in T; u\leq v\}$ be the subtree of $\TT$ rooted at $u$.  

The collection of all the Galton-Watson trees is equipped with the probability measure  $\mathbb{P}$, and with $\mathbb{E}$ as the corresponding expectation. Since $\TT$ has a positive probability to be infinite, we let $\mathbb{P}^*$ be the measure conditioned on survival, that is 
\begin{align*}
    \mathbb{P}^*(\cdot)=\mathbb{P}(\cdot| \TT \text{ is infinite}).
\end{align*}

Now we define the $\lambda$-baised random walk $(X_n)_{n\geq 0}$ on a fixed tree $\TT$. The walk starts from the root $\rho$, with the transition probability as follows:
\begin{align*}
    &P^{\TT}(X_{n+1}=u^*|X_n=u)=\dfrac{\lambda}{\lambda+\nu(u)};\\
    &P^{\TT}(X_{n+1}=v|X_n=u)=\dfrac{1}{\lambda+\nu(u)},\quad \forall v\in c(u). 
\end{align*}
For the root, we add an extra parent $\rho^*$ to it and suppose that the random walk is reflected at $\rho^*$. We let $P^\TT$ (resp. $E^\TT$) denote the law (resp. expectation) of the random walk conditioned on $\TT$, which is called the quenched law. Another measure $P$, often referred to as the annealed law, is obtained by averaging the quenched law over $\mathbb{P}$. We set $E$ as the annealed expectation. Also, we denote by $P^*$ (resp. $E^*$) the annealed probability (resp. expectation) conditioned on non-extinction. 

On the topic of $\lambda$-baised random walk on a Galton-Watson tree, it is first proved in \cite{lyons1995ergodic} that, under $P^*$, $(X_n)_{n\geq 0}$ is transient when $\lambda>m$, positive recurrent when $0<\lambda<m$ and null recurrent when $m=\lambda$. In the critical case, namely $\lambda=m$, Y.Peres and O.Zeitouni prove in \cite{peres2008central} that a central limit theorem on the height function of the walk holds, under some moments condition on the offspring distribution. This theorem is extended by A.Dembo and N.Sun in \cite{dembo2012central} to the case of multi-type Galton-Watson tree, under some weaker hypotheses. 

In \cite{aidekon2017scaling}, L.Raph\'elis and E.A\"id\'ekon consider the trace of the $\lambda$-baised walk at time $n$, that is the subtree of $\TT$ made up the vertices visited until time $n$. They consider the graph $\mathcal{R}_n:=(\TT^n,d_{\TT_n})$ where $\TT^n:=\{u\in \TT:\exists k\leq n, X_k=u\}$ is regarded as a metric space, equipped with the natural graph distance $d_{\TT_n}$. The trace can also be considered as an unlabelled tree, with the edges inherited from $\TT$. It is proved in \cite{aidekon2017scaling} that when $\lambda=m>1$, and $\nu$ has finite variation, the following convergence holds both under $P^*$ (annealed case) and under $P^\TT$ for $\mathbb{P}^*$-a,e. tree (quenched case):
\begin{align}\label{eq_convergence_BM}
    \dfrac{1}{\sqrt{\sigma^2n}}\left(\mathcal{R}_n,(X_{[nt]}\right)_{t\in [0,1]})\overset{n\rightarrow\infty}{\Longrightarrow} \left(\mathcal{T}_{(B_t)_{t\in [0,1]}},(B_t)_{t\in [0,1]}\right),
\end{align}
where $\sigma^2:={m(m-1)}/{\mathbb{E}[\nu(\nu-1)]}$. Here $(B_t)_{t\in [0,1]}$ is a reflected Brownian motion in time interval $[0,1]$ and $\mathcal{T}_{(B_t)_{t\in [0,1]}}$ is a real tree encoded by the Brownian motion. The convergence holds in law for the pointed Gromov-Hausdorff topology on metric spaces and Skorokhod topology on c\'adl\'ag functions. The concept of convergence of real trees can be found, for example, in \cite{le2005random}. 

We are interested in the case when $\nu$ does not have the second moment. The $\lambda$-baised random walk can be viewed as a random walk in random environment, where in fact no randomness is introduced in the environment. More precisely, the potential of a vertex $u$ in the first generation of $\TT$, denoted by $\triangle V(u)$, is $\ln(\lambda)$ with probability 1. The Laplace transform of the point process of the potential is defined for $t\geq0$ as
\begin{align*}
    \psi(t):=\mathbb{E}\left[\sum\limits_{|u|=1}e^{-tV(u)}\right]=\mathbb{E}\left[\frac{\nu}{\lambda^t}\right].   
\end{align*}
In \cite{de2022scaling}, L.Raph\'elis proves that under some hypothesis (the hypothesis $(\mathbf H_c)$ and $(\mathbf H_\kappa)$ in the article) on $\psi$, the height function of the walk $(|X_n|)_{n\geq 0}$ and the trace $(\mathcal{R}_n)_{n\geq 0}$ converge jointly towards the continuous-time height process of a spectrally positive stable L\'evy process of index $\kappa$ and L\'evy forest coded by it. The proof of the main theorem in that article relies on the theory about the convergence of height process and Levy trees constructed in \cite{duquesne2002random}.

The results in \cite{de2022scaling} do not cover the $\lambda$-biased random walk. On one hand, in that article, \cite[Theorem B]{goldie1991implicit} is applied as a main tool to get the regularly varying tail, which requires the non-lattice condition. On the other hand, the condition $(\mathbf H_\kappa)$ in that article is not satisfied for general $l(x)$, even when $l(x)$ is a constant. Indeed, we cannot guarantee $ \mathbb P(\nu>x){\sim} x^{-\kappa}l(x)$ and $\mathbb E[\nu^{\kappa}]<\infty$ at the same time. It is not just a technique problem, as will be shown in the Section \ref{Sec_4}, that the regularly varying tail of a key random variable comes from a totally different place. 

From now on, we assume $\lambda=m$. We also define
\begin{align*}
    a_n:=\inf\{x:P(v>x)\leq n^{-1}\},
\end{align*}
which is the proper scale as is shown in \cite[Theorem 3.7.2]{durrett2019probability}. When $l(x)=1$, we can take $a_n=n^{1/\kappa}$. 

\begin{theorem}\label{Thm_mian}
    Suppose the offspring distribution $\nu$ is regularly varying with index $\kappa$ for some $\kappa\in (1,2)$. There exists a constant $C^*\in (0,\infty)$ such that the following convergence holds in law under $P^*$ or $P^T$ for $\mathbb{P}^*$-a.e. trees:

 \begin{align*}
        \dfrac{a(n)}{n} \bigg( (|X_{[nt]}|)_{t\geq 0},\mathcal{R}_{n}\bigg)\overset{n\rightarrow\infty}{\Longrightarrow} C^*  \bigg( (H_t)_{t\geq 0}, \mathcal{T}_{(H_t)_{0\leq t\leq 1}}\bigg),
    \end{align*}
where $(H_t)_{t\geq 0}$ is the continuous-time height process of a strictly stable spectrally positive L\'evy process of index $\kappa$ and $\mathcal{T}_{(H_t)_{0\leq t\leq 1}}$ is the real tree coded by $(H_t)_{0\leq t\leq 1}$. The convergence holds in law for Skorokhod topology on c\'adl\'ag functions and the pointed Gromov-Hausdorff topology on metric spaces. 
    
    In particular, when $l(x)$ is a constant $C_l$, we have 
    \begin{align*}
         \dfrac{1}{n^{1-\frac{1}{\kappa}}} \bigg( (|X_{[nt]}|)_{t\geq 0},\mathcal{R}_{n}\bigg)\overset{n\rightarrow\infty}{\Longrightarrow} C^{**}  \bigg( (H_t)_{t\geq 0}, \mathcal{T}_{(H_t)_{0\leq t\leq 1}}\bigg),
    \end{align*}
    where $C^{**}=C^{*}C_l^{-1/\kappa}$.
\end{theorem}
\begin{remark}
 The constant $C^*$ can be computed explicitly: 
\begin{align*}
   C^*=\left(C_0|\Gamma(1-\kappa)|\right)^{-\frac{1}{\kappa}}2^{-\frac{\kappa-1}{\kappa}}\left(\frac{m-1}{m}\right)^{-\frac{2}{\kappa}},
\end{align*}
where $C_0$ is given in Proposition \ref{Prop_rvofL1}.
\end{remark}
 When $\kappa=2$, we define the truncated moment function
    \begin{align*}
        \mu(x):=E_1\left[(L^1-1)^2 1_{\{L^1\leq x\}}\right],
    \end{align*}
    where the definitions of the measure $E_1$ and the random variable $L^1$ appear in Section \ref{Sec_4}. We show in Remark \ref{Remark_kappa_2} that if $E_1[(L^1)^2]=\infty$,
    \begin{align*}
        \mu(x)\sim 2C_0 \int_{0}^x y^{-1}l(y)\d y
    \end{align*}
 which is slowly varying by \cite[Proposition 1.5.9 a]{bingham1989regular}. Note that the tail of $\mu(x)$ is different from that of $l(x)$. In fact,
    \begin{align*}
        \lim\limits_{x\rightarrow\infty} \frac{1}{l(x)}\int^x_0 t^{-1}l(t)\d t=\infty.
    \end{align*}
    We define $a^\prime(n)$ as the function of $n$ satisfying
   \begin{align*}
      \dfrac{n\mu(a^\prime(n))}{(a^\prime(n))^2}\rightarrow C_a
   \end{align*}
   for some constant $C_a$. Then we can get the parallel result for $\kappa=2$. The proof is quite the same as the case $1<\kappa<2$, so we will omit it.
   \begin{theorem}\label{Thm_2}
       Suppose the offspring distribution $\nu$ is regularly varying with index $\kappa=2$. There exists a constant $C^*\in (0,\infty)$ such that the following convergence holds in law under $P^*$ or $P^T$ for $\mathbb{P}^*$-a.e. trees:
          \begin{align*}
     \dfrac{a^\prime(n)}{n} \bigg( (|X_{[nt]}|)_{t\geq 0},\mathcal{R}_{n}\bigg)\overset{n\rightarrow\infty}{\Longrightarrow} C^*  \bigg( (|B_t|)_{t\geq 0}, \mathcal{T}_{|B_t|_{0\leq t\leq 1}}\bigg),
   \end{align*}
where $(B_t)_{t\geq 0}$ is a standard Brownian motion and 
\begin{align*}
    C^*=\left(\frac{C_a}{2}\right)^{-1/2}\frac{m}{m-1}. 
\end{align*}
In particular, when $l(x)$ is a constant $C_l$, we have 
    \begin{align}\label{eq_convergence_BM1}
         \dfrac{1}{(n\ln^{-1}(n))^{\frac{1}{2}}} \bigg( (|X_{[nt]}|)_{t\geq 0},\mathcal{R}_{n}\bigg)\overset{n\rightarrow\infty}{\Longrightarrow} C^*  \bigg( (|B_t|)_{t\geq 0}, \mathcal{T}_{|B_t|_{0\leq t\leq 1}}\bigg),
    \end{align}
with $C^*=(C_0C_l)^{-1/2}m/(m-1)$. 
   \end{theorem}
   
\begin{remark}
    When $\nu$ has the second moment, which is equivalent to $x^{-1}l(x)$ is integrable, $\mu(x)$ is bounded. Thus $a^\prime(n)=\sqrt{n}$ and $C_a=E_1[(L^1-1)^2]$. It can be easily checked that (\ref{eq_convergence_BM1}) degenerates to (\ref{eq_convergence_BM}).
\end{remark}

The article is arranged as follows. In Section 2, a global strategy is given and the details are omitted, which can be referred to in \cite{de2022scaling}. The theorem then boils down to the check of a hypothesis $(\mathbf H_l)$ in that section. In Section 3, two equivalent change of measure of the trace is introduced, which is used in the following proof. Then, the whole Section 4 is devoted to proving the hypothesis $(\mathbf H_l)$. The Subsection 4.2 is core part and the main contribution of the paper. In that section, we apply the Laplace transform and dominated convergence theorem to obtain regular variation of the tail of a key random variable and complete the proof. 

\section{Overview of the proof}
We first show an annealed version of the main theorem for random walks on forests. Denote $\mathbb{F}$ a forest made up of a collection i.i.d Galton-Watson trees $(\mathbb{T}_i)_{i\geq 1}$ with corresponding roots $(\rho_i)_{i\geq1}$. Let $(X^{\mathbb{F}}_n)_{n\geq 0}$ be a nearest-neighbour random walk  on $\mathbb{F}$ starting from $\rho_1$ with the transition probability:
\begin{align*}
    &P^{\mathbb{F}}\left(X^{\mathbb{F}}_{n+1}=v|X^{\mathbb{F}}_n=u\right)=\frac{1}{m+\nu(u)},\quad &v\in c(u);\\
    &P^{\mathbb{F}}\left(X^{\mathbb{F}}_{n+1}=u^*|X^{\mathbb{F}}_n=u\right)=\frac{m}{m+\nu(u)},\quad &u\neq \rho_i, \, \forall i;\\
    &P^{\mathbb{F}}\left(X^{\mathbb{F}}_{n+1}=\rho_{i+1}|X^{\mathbb{F}}_n=u\right)=\frac{m}{m+\nu(u)},\quad &u=\rho_i.\\
\end{align*}
The behavior of the random walk on $\mathbb{F}$ is similar to that on $\mathbb{T}$ except when it is on the roots. Let
\begin{align*}
    \mathbf{F}:=\{u\in \mathbb{F}, \exists n\geq 0, X^{\mathbb{F}}_n=u \}
\end{align*}
be the sub-forest made up with vertices once visited by the random walk. It is made up of i.i.d multi-type trees $(\mathbf{T}_k)_{k\ge 1}$. As there should be no ambiguity on the context hereafter, we will simply denote the walk by $(X_n)_{n\geq 0}$ and still denote by $\mathcal{R}_n$ the set of vertices of $\mathbb{F}$ visited up to time $n$.

\begin{proposition}\label{Prop_mian}
    Suppose ${\nu}$ is regularly varying with index $\kappa$ with $\kappa\in (1,2)$. Then the following convergence holds in law under $P$:
    \begin{align*}
        \dfrac{a(n)}{n} \bigg( (|X_{[nt]}|)_{t\geq 0},\mathcal{R}_{n}\bigg)\overset{n\rightarrow\infty}{\Longrightarrow} C^*  \bigg( (H_t)_{t\geq 0}, \mathcal{T}_{(H_t)_{0\leq t\leq 1}}\bigg)
    \end{align*}
    where $(H_t)_{t\geq 0}$ is the continuous-time height process of a strictly stable spectrally positive L\'evy process of index $\kappa$. The convergence holds in law for Skorokhod topology on c\'adl\'ag functions and the pointed Gromov-Hausdorff topology on metric spaces.  
\end{proposition}

Theorem \ref{Thm_mian} then follows from Proposition \ref{Prop_mian} by \cite[Section 5]{de2022scaling}. Note that, in \cite[Section 5]{de2022scaling}, the only requirement on $\nu$ used in the proof is $\psi((\kappa+1)/2)=m^{(1-\kappa)/2}<1$, which is obviously satisfied.

The outline of proof of Proposition \ref{Prop_mian} is the same as that in \cite{de2022scaling}. First, the trace $\mathcal{R}_n$ can be seen as multi-type Galton-Watson forest. Then $(|X_n|, \mathbf{F})$ is associated with two leafed Galton-Watson forests with edge lengths, denoted by $ \mathbf{F}^R$ and $ \mathbf{F}^X$. The height process of $ \mathbf{F}^X$  is equal to $(|X_n|)_{n\geq 0}$ and the height process of of $ \mathbf{F}^R$ is equal to that of $ \mathbf{F}$. Next, we state a result on the associated height process of such forests, under certain hypotheses. At last, we conclude Proposition \ref{Prop_mian} provided $\mathbf{F}^R$ and $\mathbf{F}^X$ satisfy the hypotheses. 

\subsection{Reduction of trees}
 In this section, we briefly introduce the reduction of trees. One can see \cite[Section 2.1]{de2022scaling} for details. For every $u\in \mathbf{F}$, which is not a root, we denote by $\beta(u)$ the edge local time of $u$:
\begin{align*}
    \beta(u):=\sharp\{n\geq 0:X_n=u^*, X_{n+1}=u \}, 
\end{align*}
which is the number of visits of $(X_n)_{n\geq 0}$ from $u^*$ to $u$. If $u$ is a root, set $\beta(u)=1$. Since in our case, $(X_n)_{n\geq 0}$ is recurrent, it's given in \cite{aidekon2017scaling} that:
\begin{lemma}\cite[Lemma 3.1]{aidekon2017scaling}
    Under the annealed law $P$, the marked forest $(\mathbf{F},\beta)$ is a multi-type Galton-Watson forest with roots of initial type 1. 
\end{lemma}
Leafed Galton-Watson forests with edge lengths are multi-type Galton-Watson forests that every vertex has two types, 0 and 1, and only vertices of type 1 may give progeny. We can denote it by a triplet $(\mathbf{F},e,l)$ where for $u\in \mathbf{F}$, $e(u)\in \{0,1\}$ stands for the type of $u$ and $l(u)\in \mathbb R^+$ is the length of edge joining $u$ with its parent. 

We can define its associated weighted height process $H^l_F$: if $u(n)$ is the $n$-th vertex of $\mathbf{F}$ in the lexicographical order, then
\begin{align*}
    H^l_F(n)=\sum\limits_{\rho\leq v\leq u(n)}l(v). 
\end{align*}
We also define the notion of the optional line of a given type. Let $\mathcal{B}^1_u$ be the set of vertices descending from $u$ in $\mathbf{F}$ having no ancestor of type $1$ since $u$. Formally, 
\begin{align*}
    \mathcal{B}^1_u:=\{ v\in F: u<v,\, \beta(w)\neq 1 ,\,\forall u<w<v \}.
\end{align*}
Also, denote by $\mathcal{L}_u^1$ the set of vertices of type $1$ descending from $u$ in $\mathbf{F}$ and having no ancestor of type $1$ since $u$. Formally,
\begin{align*}
    \mathcal{L}^1_u:=\{ v\in F: u<v,\, \beta(v)=1,\, \beta(w)\neq 1 ,\,\forall u<w<v \}.
\end{align*}
Denote by $L^1_u$ (resp. $B^1_u$) the cardinality of $ \mathcal{L}^1_u $ (resp. $\mathcal{B}^1_u$).

Now, let us briefly introduce the forests $\mathbf{F}^R$ and $\mathbf{F}^X$. Each component $\mathbf{T}_k^R$ (resp. $\mathbf{T}^X_k$) of $\mathbf{F}^R$ (resp. $\mathbf{F}^X$) is built from $\mathbf{T}_k$ of $\mathbf{F}$. We just introduce the notion without giving any intuition. $\mathbf{T}^R_k$ is built as follows:
\begin{enumerate}
    \item \textbf{Initialisation}

    Generation 0 of $\mathbf{T}^R_k$ is made up of $\rho_k$ of $\mathbf{T}_k$. Set $l(\rho_k)=0$ and $e(\rho_k)=1$.
    \item \textbf{Induction}

    Suppose that generation $n$ of $\mathbf{T}^R_k$ has been built. If generation $n$ is empty, then generation $n+1$ is empty. Otherwise, by construction, each $u\in \mathbf{T}^R_k$ with $|u|=n$ and $e(u)=1$ was associated with a vertex $u^\prime$ in $\mathbf{T}_k$. Every time we take in the lexicographical order a vertex $v^\prime\in \mathbf{T}_k$ such that $v^\prime\in \mathcal{B}^1_{u^\prime}$, we add a vertex $v$ as a child of $u$ in $\mathbf{T}^R_k$ correspondingly, thus forming the progeny of $u$. Set $e(v)=1$ if $\beta(v^\prime)$=1 and $e(v)=0$ otherwise. Let $l(v)=|v^\prime|-|u^\prime|$.
\end{enumerate}

Next, we continue to build $\mathbf{F}^X$ on the basis of $\mathbf{F}^R$. For a vertex $v$ of type $0$, which is visited by $(X_n)_{n\geq 0}$ for $k_v$ times, add $k_v-1$ siblings and attach them to the parent of $v$. They all have type 0 and the same length as $v$. For vertices $u$ of type 1, which has been visited by $(X_n)_{n\geq 0}$ from its children $k_u$ times, attach $k_u$ vertices of type $0$, length $0$ to $u$. Since each vertex in the forest corresponds to one step of $(X_n)_{n\geq 0}$, we reorder the vertices in each set of sibling according to the time of visit of the step to which they correspond. Thus we complete the construction of forest $\mathbf{F}^X$.

\subsection{Convergence of the height process associated with $\mathbf{F}^R$ and $\mathbf{F}^X$}

Let $(\mathbf{F},e,l)$ be a leafed Galton-Watson forest with edge length. We denote by $\nu(u)$ (resp. $\nu^1(u)$) the total number of children (resp. number of children of type 1) of $u\in \mathbf{F}$. Set $\mathbf{F}^1$ as the forest limited to its vertices of type $1$. We make the following hypotheses on the reproduction law:
\begin{equation*}
(\mathbf H_l)
\left\{ 
\begin{aligned}
    &(i)\, E\left[\nu^1\right]=1;\\
    &(ii)\, \exists \epsilon>0, \,\text{s.t.} \, E\left[(\nu)^{1+\epsilon}\right]<\infty;\\
    & (iii)\, \exists C_0>0, \,\text{s.t.} \, P(\nu^1>x)\sim C_0 l(x)x^{-\kappa};\\
    & (iv) \,\exists r>1, \,\text{s.t.} \, E\left[\sum\limits_{|u|=1}r^{l(u)}\right]<\infty.
\end{aligned}
\right.    
\end{equation*}
Let $m$ be the expectation of reproduction and $\mu$ be the expected sum of the length of type $1$ children, i.e.
\begin{align*}
    m:=E\left[\sum\limits_{|u|=1}1\right]=E\left[\nu\right];\quad
    \mu:=E\left[\sum\limits_{|u|=1,e(u)=1}l(u)\right].
\end{align*} 
Also, define for $\forall n\in \mathbb{N}$,
\begin{align*}
    H^1_F(n):=|u^1(n)|,\quad
    H^l_F(n):=\sum\limits_{\rho\leq v\leq u(n)}l(v).
\end{align*}
where $u(n)$ (resp $u^1(n)$) is the $n$-th vertex of $\mathbf{F}$ (resp. $\mathbf{F}^1$) taken in the lexicographical order. The next theorem shows the convergence of $H^l_F$ and $H^1_F$ under $(\mathbf{H}_l)$.
\begin{theorem} \cite[Theorem 2]{de2022scaling}\label{thm2}
    Let $(\mathbf{F},e,l)$ be a leafed Galton-Watson forest with edge lengths, with offspring distribution satisfying hypothesis $(\mathbf{H}_l)$. The following convergence holds in law:
        \begin{align*}
        \dfrac{a(n)}{n} \bigg( \left( H^l_F([ns])\right)_{s\geq 0}  , \left( H^1_F([ns])\right)_{s\geq 0}   \bigg) \overset{n\rightarrow\infty}{\Longrightarrow} \dfrac{1}{(C_0|\Gamma(1-\kappa)|)^{\frac{1}{\kappa}}}\bigg(\left(\mu H_{\frac{s}{m}}\right)_{s\geq 0},(H_s)_{s\geq 0}\bigg),
    \end{align*}
where $C_0$ is the constant defined in Proposition \ref{Prop_rvofL1} and $H$ is the continuous-time height process of a spectrally positive L\'evy process $Y$ with Laplace transform $E[\exp(-\lambda Y_t)]=\exp(t\lambda^{\kappa})$ for $\lambda,t>0$.
\end{theorem}
Then Proposition \ref{Prop_mian} can be proved by the deduction in \cite[ Section 2.3]{de2022scaling} based on Theorem \ref{thm2}. In summary, Theorem \ref{Thm_mian} follows if we can check hypothesis $(\mathbf H_l)$ is satisfied by $\mathbf{F}^R$ and $\mathbf{F}^X$, which will be done in Section \ref{Sec_4}. Before it, we prepare some preliminaries about the law of $(\beta(u))_{u\in \mathbf{F}}$ in the next section.

\section{Change of measure}\label{sec_com}
For $u\in \mathbf{F}$, we can easily compute that under $P^{\mathbb{T}}$ conditioned on $\beta(u)=k$, the process $(\beta(v))_{v\in c(u)}$ has the negative multinomial distribution $\boldsymbol{\beta}_k=\boldsymbol{\beta}_k(u)$, that is
\begin{align*}
    &P^{\mathbb{T}}\left((\beta(v))_{v\in c(u)}=(l_v)_{v\in c(u)}|\beta(u)=k\right)\\
    =&\left(\dfrac{m}{m+\nu(u)}\right)^k\left(\dfrac{1}{m+\nu(u)}\right)^{\sum\limits_{v\in c(u)} l_v} \dfrac{\left(k-1+\sum_{v\in c(u)l_v}\right)!}{(k-1)!\prod_{v\in c(u)}l_v!}.
\end{align*}

The formula can be understood in this way. When the random walk $X$ arrived at $u$, a coin is tossed with successful probability ${\nu(u)}/({m+\nu(u)})$. If succeed, $X$ moves to one of the children of $u$ with the same probability ${1}/({m+\nu(u)})$. Since $X$ is recurrent, it must come back to $u$ again and continue to choose the next vertex. We toss the coin repeatedly until the first failure and let $X$ move back to $u^*$, which completes the first trial. We repeat the procedure until the $k$-th failure. 

For the multi-type Galton-Watson tree $(\TT,\beta)$, we denote by $P_i$ its law with $\beta(\rho)=i$. It is calculated in \cite{aidekon2017scaling} from the distribution of $\beta$ that
\begin{align*}
    m_{ij}:=E_i[\sum\limits_{|u|=1}1_{\{\beta(u)=j \}}]=\dfrac{(i+j-1)!}{(i-1)!j!}\dfrac{m^{i+1}}{(m+1)^{i+j}}.
\end{align*}
The left and right eigenvectors of the matrix $(m_{ij})_{ij}$ associated with the eigenvalue $1$ are the $(a_i)_{i\geq 1}$ and $(b_i)_{i\geq 1}$ respectively, where $a_i:=(m-1)m^{-i}$ and $b_i:=(1-m^{-1})i$. They are normalized such that $\sum_{i\geq 1}a_i=1$ and $\sum_{i\geq 1} a_ib_i=1$.

Let $Z_n:=\sum_{|u|=n}\beta(u)$ to be the multi-type additive martingale of $(\mathbf{T},\beta)$. For every $i\geq 1$, $(Z_n/i)_{n\geq 0}$ becomes a martingale under $P_i$ with expectation $1$ with respect to the filtration generated by $(u,\beta(u))_{u\in \mathbf{T}, |u|\leq n}$ because $(b_i)_i$ is the right eigenvector of $(m_{ij})_{ij}$ associated to eigenvalue 1.

A new law $\hat{P}_i$ is introduced on marked trees with a spine using the multi-type additive martingale $Z_n/i$. Let $\hat{\boldsymbol{\beta}}_k=\hat{\boldsymbol{\beta}}_k(u)=(\hat{\beta}_v)_{v\in c(u)}$ be a probability measure with Radon-Nikodym derivative $Z_1/k$ with respect to $\boldsymbol{\beta}_k$. We construct $(\TT,\beta,(\omega_n)_{n\geq 0})$ under $\hat{P}_i$ in the following way.

\begin{enumerate}
    \item \textbf{Initialisation}

    Generation $0$ of T is made up of the root $\rho$ with type i. Also set $\omega_0=\rho$.

    \item \textbf{Induction}

Let $n\geq 0$. Suppose the tree and $(\omega_k,\beta(\omega_k))_{0\le k\le n}$ have been built. We let $\omega_n$ reproduce independently with the law $\hat{\boldsymbol{\beta}}_{\beta(\omega_n)}$. Other vertices $u$ of generation $n$ reproduce according to $\boldsymbol{\beta}_{\beta(u)}$ independently. Then choose a vertex as $\omega_{n+1}$ among children $u$ of $\omega_n$ each with probability $\beta(u)/(\sum_{v\in c(\omega_n)}\beta(v) )$.  
\end{enumerate}

We denote by $\hat{E}_i$ the expectation associated with $\hat{P}_i$. The new measure functions as the spinal decomposition of branching random walk. We call $(\omega_n)_{n\ge 0}$ vertices on the spine whose law is characterized by a Markov chain. 

\begin{proposition}\cite[Proposition 3]{de2022scaling}\label{Prop_ChangeofMeas1} The following statements holds. Under $\hat{P}_i$, the process $(\beta(\omega_k))_{k\geq 0}$ is a Markov chain on $\mathbb{N}$ with initial state $i$, and with transition probabilities $(\hat{p}_{ij})_{i,j\geq 1}$ where
    \begin{align*}
        \hat{p}_{ij}:=\frac{b_j}{b_i}m_{ij}.
    \end{align*}
\end{proposition}

Define the $\sigma$-field $\mathcal{F}_n:=\sigma((u,\beta(u))_{u\in\mathbb{T},|u|\leq n})$, we also have the multi-type many-to-one formula:
\begin{lemma} \cite[Lemma 7]{de2022scaling}
    For all $n\in \mathbb{N}^*$, let $g: \mathbb{N}^n\rightarrow \mathbb{R}$ be a positive measurable function and $\Xi_n$ a positive $\mathcal{F}_n$-measurable random variable, then 
    \begin{align*}
        E_i\left[\sum\limits_{|u|=n}\beta(u)g(\beta(u_1),\beta(u_2),\dots,\beta(u_n))\Xi_n\right]=i\hat{E}_i \left[g(\beta(\omega_1),\beta(\omega_2),\dots,\beta(\omega_n))\Xi_n\right].
    \end{align*}
\end{lemma}

The new measure $\hat{P}_i$ is a little hard to apply to our case since we choose a vertex on the spine according to its type. However, as is also introduced in \cite{de2022scaling}, we can apply another equivalent construction based on the classic spinal decomposition. 

Define $\hat{\nu}$ a random variable with Radon-Nikodym derivative $\nu/m$ with respect to the reproduction distribution $\nu$, that is for any bounded measurable function $f:\bigcup_{n\in \mathbb{N}\cup\mathbb{\{N\}} }\mathbb{R}^n\rightarrow \mathbb{R}$, we have $\mathbb{E}[f(\hat{\nu})]=\mathbb{E}[(\nu/m)f(\nu)]$. We first introduce a law $\hat{\mathbb{P}}$ on $\mathbb{T}$ which is often referred to as the spinal decomposition. 
\begin{enumerate}
    \item \textbf{Initialisation}
    
    Generation 0 is made of the root $\rho$. Also set $\tilde{\omega}_0=\rho$.

    \item \textbf{Induction}

    Let $n\geq 0$. Suppose the tree and $\tilde{\omega}_k$ up till $n$ has been built. The vertex $\omega_n$ reproduces according to $\hat{\nu}$ independently. Other vertices $u$ at generation $n$ reproduces according to $\nu$ independently. Then choose a vertex on average from the children of $\tilde{\omega}_n$ as $\tilde\omega_{n+1}$.
    
\end{enumerate}

Now we run a series of random walks on $\mathbb{T}$ under $\hat{\mathbb{P}}$. For any vertex  $\tilde{\omega}_k,k\ge 0$ on the spine, we associate it with two i.i.d truncated nearest-neighbour random walk, denoted as $(X^{1,\tilde{\omega}_k}_n)_{n\geq 0}$ and $(X^{2,\tilde{\omega}_k}_n)_{n\geq 0}$ respectively. Each of them is defined as follows. It starts on $\tilde{\omega}_k$. If it is on $u\in \mathbb{T}_{\tilde\omega_k }$, then it jump to one of the child with probability ${1}/({m+\nu(u)})$ and to the parent of $u$ with probability ${m}/({m+\nu(u)})$. If it reach $\tilde{\omega}_{k-1}$, then it is killed at once.

Let for each $u\in \mathbb{T}$, 
\begin{align*}
    \tilde{\beta}_k^j(u):=\sharp\{n\geq 0: X^{j,\tilde{\omega}_k}_n=u^*, X^{j,\tilde{\omega}_k}_{n+1}=u \},\,k\geq 0, j=1,2
\end{align*}
be the edge local times on $u$ of the walk launched on $\tilde{\omega}_k$, and let
\begin{align*}
    \tilde{\beta}^j(u):=\sum\limits_{k=1}^\infty \tilde{\beta}_k^j(u),\, j=1,2.
\end{align*}
Finally let
\begin{align*}
     \tilde{\beta}(u):= \tilde{\beta}^1(u)+ \tilde{\beta}^2(u)+1_{\{u\in(\tilde{\omega}_k)_{k\geq 0}  \}}.
\end{align*}
 We set $\tilde{T}:=\{u\in \mathbb{T}, \tilde{\beta}(u)\geq 0\}$. It is proved in \cite{de2022scaling} that the two changes of measure are indeed the same. 
\begin{proposition}\cite[Proposition 5]{de2022scaling}\label{eq-change}
    Under $\hat{P}_1$, $(T,\beta,(\omega_k)_{k\geq 0})$ has the same law as $(\tilde{T},\tilde{\beta},(\tilde{\omega}_k)_{k\geq 0})$ with $(\tilde{T},(\tilde{\omega}_k)_{k\geq 0})$ built under $\hat{\mathbb{P}}$.
\end{proposition}

\section{Proof of Proposition 2}\label{Sec_4}
The key of the proof is to ensure that $\mathbf{F}^R$ and $\mathbf{F}^X$ satisfy $(\mathbf H_l)$. By construction, for both $\mathbf{F}^R$ and $\mathbf{F}^X$, $\nu^1(u)$ has the law of $L^1$, the cardinality of the optional line $\mathcal{L}^1$, under $P_1$. The law of the total offspring distribution of $\mathbf{F}^R$ has the law of $B^1$, the cardinal of $\mathcal{B}^1$. The law of the total offspring distribution of $\mathbf{F}^X$ has the law of $\sum_{u\in \mathcal{B}^1 }2\beta(u)$, the total time spent in $\mathcal{B}^1$ in one excursion. 

Therefore, it's enough to verify that $E_1[L^1]=1$, which proves $(\mathbf H_l)$(i) of $\mathbf{F}^R$ and $\mathbf{F}^X$; and verify that $E_1[(\sum_{u\in \mathcal{B}^1}\beta(u))^{1+\epsilon}]<\infty$, which proves $(\mathbf H_l)$(ii) of $\mathbf{F}^R$ and $\mathbf{F}^X$. For $(\mathbf H_l)$(iv), notice that each vertex in first generation $\mathbf{F}^R$, denoted by $v^R$, corresponds to a vertex $v\in \mathcal{B}^1$ and that $l(v^R)=|v|$. Moreover, a vertex of the first generation of $\mathbf{F}^X$, denoted as $v^X$, corresponds to a vertex $v^R$ of $\mathbf{F}^R$ after having replicated itself a number of $2\beta(u)-1$ of times. Hence, 
\begin{align*}
    E\left[\sum\limits_{u\in \nu^R}r^{l(u)}\right ]=E_1\left [\sum\limits_{u\in \mathcal{B}^1}r^{|u|}\right ]\leq E_1\left [\sum\limits_{u\in \mathcal{B}^1}2\beta(u)r^{|u|}\right ]=E\left [\sum\limits_{u\in \nu^X}r^{l(u)}\right ].
\end{align*}
Therefore, we only need to verify that there exists an $r>1$ such that $E_1\left [\sum_{u\in \mathcal{B}^1}2\beta(u)r^{|u|}\right ]<\infty$. 

Finally, for $(\mathbf H_l)$(iii), we need to check that there exists a positive constant $C_0$ such that $P_1(L^1> x)\sim C_0x^{-\kappa}$. This is the core of the article as well as the main contribution. The proof is quite different from that in \cite{de2022scaling}, since in our case the regular varying tail of $L^1$ comes from the tail behavior of $\nu$, instead of that of $\mathcal{L}^1$. 

\subsection{Hypotheses ($\mathbf H_l$)(i),(ii) and (iv)}
We define 
\begin{align*}
    \hat{\tau}_1:=\min\{k\geq 1: \beta(\omega_k)=1 \}
\end{align*}
as the first non-null hitting time of $1$ by the Markov chain $(\beta(\omega_k))_{k\geq 0}$. Hypotheses ($\mathbf H_l$)(i),(ii) and (iv) then follows from Lemma \ref{Lm_exp_L1}, \ref{Lm_sumbeta} and \ref{Lm_moment_B1}.

\begin{lemma}\label{Lm_exp_L1}
    For any $i\geq 1$, we have,
    \begin{align*}
        E_i[L^1]=i.
    \end{align*}
    In particular, $E_1[L^1]=1$, and the reproduction law of both $\mathbf{F}^R$ and $\mathbf{F}^X$ satisfies conditions $(\mathbf H_l)$(i).
\end{lemma}
\begin{proof}
    The proof is the same as that of \cite[Lemma 9]{de2022scaling}. For readers' convenience, we spell out the details. For any $i\geq 1$, we have
    \begin{align*}
        E_i\left[L^1\right]&=E_i\left[\sum\limits_{u\in \mathcal{L}^1}1\right]=E_i\left[\sum\limits_{u\in \mathcal{L}^1}\beta(u)\right]\\
        &=\sum\limits_{k\geq 1}E_i\left[\sum\limits_{|u|=k}\beta(u)1_{\{\beta(u_1),\dots ,\beta(u_{k-1})\neq 1,\beta(u)=1\}}\right]\\
        &=\sum\limits_{k\geq 1} i\hat{E}_i\left[1_{\{\beta(\omega_1),\dots,\beta(\omega_{k-1})\neq 1,\beta(\omega)=1\}}\right]\\
        &=i\sum\limits_{k\geq 1} \hat{E}_i\left[1_{\{\hat{\tau}_1=k\}}\right]=i.
    \end{align*}
     where we used the multi-type many-to-one lemma in the last but one equation.
\end{proof}

\begin{lemma}\label{Lm_sumbeta}
	For every $\alpha>0$, there exist $C_\alpha>0$ such that for any $r\in (1,m^\alpha)$ and $i\geq 1$, 
	\begin{equation*}
		\hat{E}_i\left[\sum_{k=1}^{\hat{\tau}_1}(\beta(\omega_k))^\alpha r^k\right]\leq C_\alpha i^\alpha.
	\end{equation*}
	As a consequence, first there exists a constant $C_p>0$, such that for any $p>0$
	\begin{equation*}
		\hat{E}_i\left[\hat{\tau}_1^p\right]\leq C_p\ln^p(1+i)
	\end{equation*}
and second the laws of $\mathbf{F}^R$ and $\mathbf{F}^X$ satisfy $(\mathbf H_l)$(iv).
\end{lemma}
\begin{proof}
	The proof is similar to \cite[Lemma 10]{de2022scaling}. Note that here we can take $\alpha$ as any positive real number, since in that proof
 \begin{align*}
     E\left[\sum_{|u|=1} e^{-(\alpha+1)V(u)}\right]=E\left[\sum_{|u|=1} m^{-(\alpha+1)}\right]=m^{-\alpha}<1
 \end{align*}
for any $\alpha>0$.
\end{proof}

\begin{lemma}\cite[Lemma 11]{de2022scaling}\label{Lm_moment_B1}
 We set $\tilde{B}^1:=\sum_{u\in \mathcal{B}^1}\beta(u)$. For every $\alpha\in (0,\kappa-1)$, $\epsilon>0$, there exists a constant $C^\prime_{\alpha+\epsilon} >0$ such that for any $i\geq 1$,
    \begin{align*}
        E_i[(\tilde{B}^1)^{1+\alpha}]\leq C^\prime_{\alpha+\epsilon} i^{1+\alpha+\epsilon}.
    \end{align*}
    As a consequence, the law of $\mathbf{F}^R$ and $\mathbf{F}^X$ satisfy hypothesis $(\mathbf H_l)$(ii).
\end{lemma}

\subsection{Regular varying tail of $L^1$}
Recall that we have assumed that $\kappa\in(1,2)$  and the function $l(x)$ is slowly varying. Instead of proving the regular varying tail of $L^1$ under $P_1$, we prove it under $\hat{P}_1$ thanks to Lemma \ref{Lm_CM}.

\begin{lemma}\cite[Lemma 13]{de2022scaling}\label{Lm_CM}
As $x\rightarrow\infty$, $P_1(L^1>x)\sim x^{-\kappa}l(x)$ if and only if as $x\rightarrow \infty$, 
\begin{align*}
    \hat{P}_1(L^1>x)\sim\frac{\kappa}{\kappa-1}x^{-(\kappa-1)}l(x).
\end{align*}
\end{lemma}
We then claim the main proposition of the section. 

\begin{proposition}\label{Prop_rvofL1}
	We have 
	\begin{equation*}
	\hat{P}_1(L^1>x)\sim C_\kappa x^{-(\kappa-1)}l(x),
	\end{equation*} 
	where 
 \begin{align*}
      C_\kappa=\frac{2\Gamma(\kappa)\kappa m}{(m-1)^3(m+1)^{\kappa-2}(\kappa-1)}.
 \end{align*}
As a result, thanks to Lemma \ref{Lm_CM}, $(H_l)$(iii) holds with 
\begin{align*}
    C_0=\frac{\kappa-1}{\kappa}C_{\kappa}=\frac{2\Gamma(\kappa)m}{(m-1)^3(m+1)^{\kappa-2}}.
\end{align*}

 \end{proposition}
\begin{remark}
    \label{Remark_kappa_2}
    For $\kappa=2$, we can prove in instead that when $E_1[(L^1)^2]=\infty$,
    \begin{align*}
       \mu(x)=E_1\left[(L^1-1)^2\textbf{1}_{\{L^1\leq x\}}\right]\sim \hat{E}_1\left[L^1\textbf{1}_{\{L^1\leq x\}}\right].
    \end{align*}
    In fact, it holds that
\begin{equation*}
\begin{split}
        E_1\left[(L^1-1)^2\textbf{1}_{\{L^1\leq x\}}\right]&\le \sum_{k\geq1}E_1\left[\sum_{|u|=k}\textbf{1}_{\{u\in\mathcal{L}^1\}}E_1\left[L^1\textbf{1}_{\{L^1\leq x\}}|\mathcal{F}_k\right]\right]\\
    &\le \hat{E}_1\left[L^1\textbf{1}_{\{L^1\leq x\}}\right]
\end{split}
\end{equation*}
and 
\begin{align*}
          E_1\left[(L^1-1)^2\textbf{1}_{\{L^1\leq x\}}\right]&\ge \sum_{k\geq1}E_1\left[(\sum_{|u|=k}\textbf{1}_{\{u\in\mathcal{L}^1\}}-2) E_1\left[L^1\textbf{1}_{\{L^1\leq x\}}|\mathcal{F}_k\right]\right]\\
    &\ge \hat{E}_1\left[L^1\textbf{1}_{\{L^1\leq x\}}\right]-2.
\end{align*}
We can find an analog of Proposition \ref{Prop_rvofL1} when $\kappa=2$ that
\begin{align*}
    \hat{E}_1\left[L^1\textbf{1}_{\{L^1\leq x\}}\right]\sim 2C_0 \int_{0}^x y^{-1}l(y)\d y
\end{align*}
with a little modification.
\end{remark}

Recall that $(\beta(\omega_k))_{k\geq 0}$ is Markov chain with transition matrix $(\hat{p}_{ij})_{ij}$ under $\hat{P}_1$. In order to obtain the tail behavior of $L^1$, we need the second construction mentioned in Section \ref{sec_com}. Under the measure $\hat{\mathbb{P}}$, $\omega_{k}$ is chosen among the children of $\omega_{k-1}$ on average. Let $\Omega(\omega_{k})$ be the siblings of $\omega_k$. If we condition on $\beta(\omega_{k-1})$ and $\beta(\omega_{k})$, it is easy to see that the distribution of $(\beta(u))_{u\in \Omega(\omega_{k})}$ does not depend on anything else.

In detail, conditioned on $\beta(\omega_{k-1})$ and $\beta(\omega_{k})$, there are $\beta(\omega_{k-1})-1$ times of visit of the directed edge $(\omega_{k-2},\omega_{k-1})$ by the random walks $X_l^{1,\omega_{l}}$ and $X_l^{2,\omega_{l}}$ where $0 \leq l\leq k-2$. ($1$ is subtracted since $\omega_{k-1}$ is on the spine so an additional time is counted.) Moreover, there are two other random walks $X_n^{1,\omega_{k-1}}$ and $X_n^{2,\omega_{k-1}}$ which start from $\omega_{k-1}$ that killed when hitting $\omega_{k-2}$. In all, there are $\beta(\omega_{k-1})+1$ times that a random walk starts from $\omega_{k-1}$. Each time, we toss a coin with probability $m/(\nu(\omega_{k-1})+m)$ to succeed. If it succeeds, we choose one of the children of $\omega_{k-1}$ with the probability ${1}/({\nu(\omega_{k-1})+m})$ and we repeat the procedure until it fails for $\beta(\omega_{k-1})+1$ times. 

Since $\omega_k$ is chosen on average from the children of $\omega_{k-1}$, without loss of generality, we assume it is the first child of $\omega_{k-1} $. Therefore, the conditional distribution of $(\beta(u),u\in \Omega(\omega_k))$, denoted as $(\beta_2,\beta_3,\dots\beta_{\nu(\omega_{k-1})})$, given $\beta(\omega_{k-1}),\beta(\omega_k)$ is
\begin{align*}
&\hat{P}_1\left((\beta_2,\beta_3,\dots\beta_{\nu(\omega_{k-1})})=(b_2,b_3,\dots b_{\nu(\omega_{k-1})})\Bigg|\beta(\omega_{k-1}),\beta(\omega_{k})\right)\\
=&\dfrac{\left(\beta(\omega_{k-1})+\beta(\omega_{k})+\sum_{j=2}^{\nu(\omega_{k-1})}b_{j}-1\right)!}{\left(\beta(\omega_{k-1})\right)!\left(\beta(\omega_{k})-1\right)!\prod_{j=2}^{\nu(\omega_{k-1})} (b_j)!}\dfrac{m^{\beta(\omega_{k-1})+1}}{\left(\nu(\omega_{k-1})+m\right)^{\beta(\omega_{k-1})+\beta(\omega_{k})+\sum_{j=2}^{\nu(\omega_{k-1})}b_j}}\\
	 &\left(\dfrac{(\beta(\omega_{k-1})+\beta(\omega_{k})-1)!}{(\beta(\omega_{k-1}))!(\beta(\omega_{k})-1)! }\dfrac{(m)^{\beta(\omega_{k-1})+1}}{(1+m)^{\beta(\omega_{k-1})+\beta(\omega_{k})}}\right)^{-1}\\
	=&\dfrac{\left(\beta(\omega_{k-1})+\beta(\omega_{k})+\sum_{j=2}^{\nu(\omega_{k-1})}b_j-1\right)!}{\left(\beta(\omega_{k-1})+\beta(\omega_{k})-1\right)!\prod_{j=2}^{\nu(\omega_{k-1})} (b_j)!}\dfrac{(m+1)^{\beta(\omega_{k-1})+\beta(\omega_{k})}}{(\nu(\omega_{k-1})+m)^{\beta(\omega_{k-1})+\beta(\omega_{k})+\sum_{j=2}^{(\omega_{k-1})}b_j}}.
\end{align*}
It is exactly negative multinomial distribution with parameters $\beta(\omega_{k-1})+\beta(\omega_k)$ and $1/(m+\nu(\omega_{k-1})) $. In other words, at $\omega_{k-1}$, we toss a coin with successful probability $(\nu(\omega_{k-1})-1)/(m+\nu(\omega_{k-1}))$. When succeed, we choose one sibling of $\omega_k$ on average and let the random walk move to it. Toss the coin repeatedly until the first failure, which completes one trial. We repeat the procedure until the $(\beta(\omega_{k-1})+\beta(\omega_k) )$-th failure. 

We denote such probability measure by $\tilde{P}_\beta$, that is $\tilde{P}_\beta(\cdot):=\hat{P}_1(\cdot|\beta(\omega_{k-1})+\beta(\omega_k)=\beta)$. Note that $\tilde{P}_\beta(\cdot)$ does not depend on $k$ and $\beta_k$, $\beta_{k-1}$. Hence we may and will assume that $k=1$ when dealing with $\tilde{P}_\beta(\cdot)$. 

Let $(u_j)_{2\leq j\leq \nu(\rho)}$ be the siblings of $\omega_1$. Define the random variable $L^1_0:=\sum_{j=2}^{\nu(\rho)}L^1_{u_j}$, where $L^1_{u_j}$ is the number of type 1 children of the root in the subtree $\TT_{u_j}$. We see that $L^1_{u_j}$ is independent of each other and only depends on $\beta(u_j)$.

\begin{proposition}\label{Prop_main}
The random variable $L_0^1$ is regularly varying, namely
	\begin{equation*}
	\tilde{P}_\beta(L_0^1>x)\sim \frac{C^\prime_\kappa \beta }{m(m+1)^{\kappa-1}}l(x)x^{-(\kappa-1)},
	\end{equation*}
	where $C^\prime_\kappa={\Gamma(\kappa)\kappa}/{(\kappa-1)}$.
\end{proposition}
In order to prove the Proposition \ref{Prop_main}, we need fine understanding of the behaviour of $(L^1_{u_j})_{2\leq j\leq \nu(\rho)}$. Under $\tilde{P}_\beta$, for $u$ as a sibling of $\omega_1$, $\beta(u)$ has the law of the sum of $\beta$ geometric random variables of parameter ${1}/({m+\nu(\rho)})$. Denote by $N+1$ the number of children of the root $\rho$. Let $V:=\sum_{j=2}^{N+1} \beta_j$ be the sum of edge local times of siblings of $\omega_1$, where $\beta_j:=\beta(u_j)$. The next two lemmas show that the regular varying tail of $L^1_0$ comes from $V$ rather than from $(L^1_{u_j})_{2\leq j\leq N+1}$ as in \cite{de2022scaling}.

\begin{lemma}\cite[Lemma 14]{de2022scaling}\label{Lm_MomentofL}
	Let $\alpha\in(0,\kappa-1)$, and $i\geq 1$. There exist a constant $C_\alpha>0$ such that 
	\begin{equation*}
	E_i[(L^1)^{1+\alpha}]\leq C_\alpha i^{1+\alpha}.
	\end{equation*}
\end{lemma}

 From lemma \ref{Lm_MomentofL}, we know that for $u\in\Omega(\omega_1)$, $L^1_u$ has finite moment of order $1+\alpha$ for $\alpha\in(0,\kappa-1)$, since a vertex $u$ not on the spine reproduce with the probability measure $P_{\beta(u)}$.

\begin{lemma}\label{Lm_Sumofbeta}
	For $V$ defined above, we have
	\begin{align*}
		\tilde{P}_\beta(V>x)\sim \frac{C^\prime_\kappa \beta }{m(m+1)^{\kappa-1}}l(x)x^{-(\kappa-1)},
	\end{align*}
	where $C^\prime_\kappa={\Gamma(\kappa)\kappa}/{(\kappa-1)}$.
\end{lemma}
\begin{proof}
	First since the tail distribution of $\nu(\rho)$ is regularly varying and $N=\nu(\rho)-1$, by the Karamata theorem (see, for example, \cite[Theorem 8.1.4]{bingham1989regular}) and many-to-one formula, we have
	\begin{align}\label{eq4.9}
	\hat{\mathbb{P}}(N>x)=\mathbb{E}\left[\frac{\nu}{m}\textbf{1}_{\{N>x\}}\right]\sim \frac{\kappa}{m(\kappa-1)}x^{-(\kappa-1)}l(x).
	\end{align}
	
	Since the $\beta$ trials are taken independently and $V=\sum_{k=1}^{\beta}V^k$, where $V^k$ is the sum of edge local times in the $k$-th trial. We only need to prove that 
	\begin{align*}
		\tilde{P}_1(V>x)\sim \frac{C^\prime_\kappa}{m}\left(\dfrac{1}{m+1}\right)^{\kappa-1}l(x)x^{-(\kappa-1)}.
	\end{align*}
    In fact, since $V^k$ are i.i.d random variable, with regular varying tail of the same order, by Laplace transform \cite[Theorem 8.1.6]{bingham1989regular}, we obtain that the sum of them is also regularly varying with the tail in (\ref{eq4.9}).	
	
    In one trial, when $N$ fixed, $V$ follows the geometric distribution with the failure probability ${(m+1)}/{(m+N+1)}$, namely,
	\begin{align*}
		\tilde{P}_1(V>n|N)=\left(1-\frac{m+1}{m+N+1}\right)^{n+1}.
	\end{align*}
	By averaging over $N$, we get that when $n$ large enough,
	\begin{align*}
		&\tilde{P}_1(V>n)=\hat{\mathbb{E}}[(\tilde{P}_1(V>n|N)]\\
		=&-\int_{x=0}^{\infty}\left(1-\frac{m+1}{m+x+1}\right)^{n+1} \d\hat{\mathbb{P}}(N> x)\\
		=&\int_{x=0}^{\infty}(n+1)\left(1-\frac{m+1}{m+x+1}\right)^{n}\frac{m+1}{(m+x+1)^2}\hat{\mathbb{P}}(N> x) \d x\\
		\sim&\frac{\kappa(m+1)}{m(\kappa-1)}\int_{x=0}^{\infty}n\left(1-\frac{m+1}{m+x+1}\right)^{n}(l(x)x^{-\kappa-1}\wedge 1) \d x,
	\end{align*}
 where we apply (\ref{eq4.9}) in the last line.
 
To calculate the integral, we let $({m+x+1})/({m+1})$ be the new variable. Since $l(x)$ is slowly varying, we see
	\begin{align*}
		&\frac{\kappa(m+1)}{m(\kappa-1)}\int_{x=0}^{\infty}n\left(1-\frac{m+1}{m+x+1}\right)^{n}(l(x)x^{-\kappa-1} \wedge 1)\d x\\
		\sim&\frac{\kappa}{m(\kappa-1)}(m+1)^{-(\kappa-1)}\int_{x=1}^{\infty}n\left(1-\frac{1}{x}\right)^{n}l(x)x^{-\kappa-1} \d x\\
		\sim&\frac{\kappa}{m(\kappa-1)}(m+1)^{-(\kappa-1)}\int_{x=0}^{1}n(1-x)^{n}l\left(\frac{1}{x}\right)x^{\kappa-1} \d x.
	\end{align*}
	Then we evaluate the tail of $\int_{x=0}^{1}n(1-x)^{n}l({1}/{x})x^{\kappa} \d x$. By changing of variable again,
	\begin{align*}
		\frac{n^{\kappa-1}}{l(n)}\int_{x=0}^{1}n(1-x)^{n}l\left(\frac{1}{x}\right)x^{\kappa-1} \d x&=\int_{x=0}^{1}(1-x)^{n}\frac{l\left(\frac{1}{x}\right)}{l(n)}(nx)^{\kappa-1}\d nx\\
		&=\int_{x=0}^{n}\left(1-\frac{x}{n}\right)^{n}\dfrac{l\left(\frac{n}{x}\right)}{l(n)}x^{\kappa-1}\d x
	\end{align*}
	Note that $(1-\frac{x}{n})^{n}\leq e^{-x}$ and for $0<x<n$,
 \begin{align*}
     \dfrac{l\left(\frac{n}{x}\right)}{l(n)}\leq C_l \exp\left\{\int_{n}^{\frac{n}{x}}\epsilon(u)\d u/u\right \} 
 \end{align*}
 from the representation of slowly varying function \cite[Theorem 1.3.1]{bingham1989regular}, where $C_l$ is a constant determined by $l(x)$ and $\epsilon(u)$ goes to $0$ as $u\rightarrow\infty$. We claim that, we can find $\varepsilon>0$ small enough and a constant $C_\varepsilon$ such that for $n$ large enough
 \begin{align*}
     \exp\left\{\int_{n}^{\frac{n}{x}}\epsilon(u)\d u/u\right \} \le C_\varepsilon(x^{\varepsilon}\vee x^{-\varepsilon}).
 \end{align*}
Indeed, since  $\epsilon(u)$ goes to $0$ as $u\rightarrow\infty$, we can find $M,\varepsilon>0$ such that for $u>M$, $|\epsilon(u)|<\varepsilon$. We first consider the case $0<x<1$.  and thus when $n\ge M$,
\begin{align*}
     \exp\left\{\int_{n}^{\frac{n}{x}}\epsilon(u)\d u/u\right \}\le  \exp\left\{\int_{n}^{\frac{n}{x}}\varepsilon\d u/u\right \}\le x^{-\varepsilon}.
\end{align*}
Second, if $1\le x\le \sqrt{n}$, when $n\ge M^2$,
\begin{align*}
     \exp\left\{\int_{n}^{\frac{n}{x}}\epsilon(u)\d u/u\right \}\le  \exp\left\{\int^{n}_{\frac{n}{x}}\varepsilon\d u/u\right \}\le x^{\varepsilon}.
\end{align*}
Finally, if $\sqrt{n}<x<n$,
\begin{align*}
    \exp\left\{\int_{n}^{\frac{n}{x}}\epsilon(u)\d u/u\right \}&\le \exp\left\{\int^{M}_{1}\epsilon(u)\d u/u\right \} \exp\left\{\int^{n}_{M}\epsilon(u)\d u/u\right \}\\
    &\le \exp\left\{\int^{M}_{1}\epsilon(u)\d u/u\right \} n^{\varepsilon}\le \exp\left\{\int^{M}_{1}\epsilon(u)\d u/u\right \} x^{2\varepsilon}.
\end{align*}
Since $\varepsilon$ is arbitrary, we complete the proof of the claim.
 
 To conclude, we get that when $n$ is large enough
	\begin{align*}
		\left(1-\frac{x}{n}\right)^{n}\dfrac{l\left(\frac{n}{x}\right)}{l(n)}x^{\kappa-1}\leq C_lC_\varepsilon(x^{\kappa-1-\varepsilon}\vee x^{\kappa-1+\varepsilon})e^{-x} ,
	\end{align*}
	which is integrable on $(0,\infty)$. Therefore, by the dominated convergence theorem,
	\begin{align*}
		\lim\limits_{n\rightarrow\infty}\int_{x=0}^{n}\left(1-\frac{x}{n}\right)^{n-1}\dfrac{l\left(\frac{n}{x}\right)}{l(n)}x^{\kappa-1}\d x=\int_{x=0}^{\infty}e^{-x}x^{\kappa-1}\d x=\Gamma(\kappa).
	\end{align*}
 Thus we have 
 \begin{align*}     
    \frac{\kappa}{m(\kappa-1)}(m+1)^{-(\kappa-1)}\int_{x=1}^{\infty}n(1-\frac{1}{x})^{n}l(x)x^{-\kappa-1} \d x
	\sim\frac{C^\prime_\kappa}{m}(m+1)^{-(\kappa-1)}n^{-(\kappa-1)}l(n),
 \end{align*}
 where $C^\prime_\kappa={\Gamma(\kappa)\kappa}/{(\kappa-1)}$.
\end{proof}

Let $N_i:=\sum_{2\leq j\leq N+1}\textbf{1}_{\beta(u_j)=i}$ be the number of type $i$ vertices in $\Omega(\omega_1)$. We want to estimate $N_i$ conditioned on $N$ and $V$. Let $a_i={N_i}/{N}$, we immediately see that
\begin{align*}
&\sum_{i=0}^{\infty}a_i=1;\\
&\theta:=\sum_{i=0}^{\infty}ia_i=\frac{V}{N}.
\end{align*}
The next lemma says that the sum of local times of the vertices which are visited for many time is relatively small. Given $N$ and $V$, $\beta$ provides no more information so we omit the subscript.
\begin{lemma}\label{Lm_Klarge}
	For any $\epsilon$>0, we can find some $K=K(\theta,\epsilon)\in\mathbb{N}^*$ such that 
	\begin{equation*}
	\tilde{P}(\sum_{i=K}^{\infty}ia_i\geq \epsilon|N,V)\leq \epsilon.
	\end{equation*} 
\end{lemma}
\begin{proof}
	By Markov inequality, it holds that
	\begin{align*}
	\tilde{P}\left(\sum_{i=K}^{\infty}ia_i\geq \epsilon|N,V\right)\leq \frac{\tilde{E}\left[\sum_{i=K}^{\infty}iN_i|N,V\right]}{N\epsilon}.
	\end{align*}
	Since $\tilde{E}[\sum_{i=K}^{\infty}iN_i|N,V]$ is the sum of edge local times of the vertices which are visited more than $K$ times,
	\begin{align*}
	\tilde{E}\left[\sum_{i=K}^{\infty}iN_i|N,V\right]=\tilde{E}\left[\sum_{j=2}^{N+1}\beta_j1_{\{\beta_j\geq K\}}\Bigg|N,V\right]=N\tilde{E}\left[\beta_21_{\{\beta_2\geq K\}}|N,V\right].
	\end{align*}
	For $\beta_2$ conditioned on $N,V$, it equals $\sum_{i=1}^{V} \Xi_i$ where $\Xi_i$ are i.i.d binomial distribution with probability ${1}/{N}$ to be $1$. Then we can easily deduce that 
	\begin{align*}
	\tilde{E}[\beta_2^2|N,V]=\tilde{E}\left[\sum_{i=1}^V \Xi_i|N,V\right]+\tilde{E}&\left[\sum_{i\neq j}\Xi_i\Xi_j\right]=\frac{V}{N}+\frac{V(V-1)}{N^2}\leq \theta+\theta^2.
	\end{align*}
	Thus, by Markov inequality again,
	\begin{align*}
	\tilde{E}[\beta_2 1_{\{\beta_2\geq K\}}|N,V]\leq \frac{\tilde{E}[\beta_2^2|N,V] }{K}\leq \frac{\theta+\theta^2}{K}
	\end{align*}
	Take $K\geq {(\theta^2+\theta)}/{\epsilon^2}$ and then $\tilde{E}[\beta_21_{\{\beta_2\geq K\}}|N,V]\leq\epsilon^2$. Therefore
	\begin{equation*}
	\tilde{P}\left(\sum_{i=K}^{\infty}ia_i\geq \epsilon|N,V\right)\leq \frac{\tilde{E}[\beta_11_{\{\beta_1\geq K\}}|N,V]}{\epsilon}\leq\epsilon.
	\end{equation*}
\end{proof}

Now we can prove the proposition \ref{Prop_main}.
\begin{proof}[Proof of Proposition \ref{Prop_main}]
	By \cite[Theorem 8.1.6]{bingham1989regular}, it suffices to prove that 
	\begin{align*}
	1=&\liminf\limits_{s\rightarrow 0+}\frac{c_{\kappa}m(m+1)^{\kappa-1} s^{-(\kappa-1)}}{l\left(\frac{1}{s}\right)\beta C_\kappa^\prime}\left(1-\tilde{E}_\beta\left[e^{-sL_0^1}\right]\right)\\=&\limsup\limits_{s\rightarrow 0+}\frac{c_{\kappa}m(m+1)^{\kappa-1} s^{-(\kappa-1)}}{ l\left(\frac{1}{s}\right)\beta C_\kappa^\prime}\left(1-\tilde{E}_\beta\left[e^{-sL_0^1}\right]\right).
	\end{align*}
Here $c_\kappa=\Gamma(2-\kappa)$ is a constant appearing when we apply the Laplace transform. Thus, we only need to prove the upper limit and the lower limit are the same.
	
	\textit{1.Lower limit}
	
First, by following the proof of Lemma \ref{Lm_Sumofbeta}, we claim that we can take $A$ large enough such that for $x$ large enough,
	\begin{align}\label{eq_VAN}
		\tilde{P}_\beta(V>x,V\leq AN)\geq(1-\epsilon)\tilde{P}_\beta(V>x).
	\end{align}
In fact, if $V>x\vee AN$, among the $\beta$ trials, at least in one trial $V>(x\vee AN)/\beta$, which yields 
	\begin{align*}
		\tilde{P}_\beta(V>x\vee AN)\leq \beta \tilde{P}_1(\beta V>x\vee AN).
	\end{align*}
When $N\geq {x}/{A}$, 
\begin{align*}
    \tilde{P}_1(\beta V> AN|N)\leq \left(1-\frac{m+1}{m+N+1}\right)^{\frac{AN}{\beta}}\leq e^{-\frac{A}{\beta(m+1)}}.
\end{align*}
Then there exists a constant $C,\epsilon>0$ such that
	\begin{align*}
		\tilde{P}_1\left(\beta V> AN,N\geq \frac{x}{A}\right)\leq C  e^{-\frac{A}{\beta(m+1)}}\left(\frac{x}{A}\right)^{-(\kappa-1)}l\left(\frac{x}{A}\right)\le \frac{\epsilon}{2\beta}\tilde{P}_\beta(V>x) 
	\end{align*}
 when $A$ large enough by Lemma \ref{Lm_Sumofbeta}. 
 
 When $N< {x}/{A}$, we follow in the steps of Lemma \ref{Lm_Sumofbeta}
	\begin{align*}
		&\tilde{P}_1\left(\beta V> x,N< \frac{x}{A}\right)=-\int_{y=0}^{\frac{x}{A}}\left(1-\frac{m+1}{y+m+1}\right)^{\frac{x}{\beta}}\d \hat{\mathbb{P}}(N>y)\\
		=&o(x^{-(\kappa-1)}l(x))+\dfrac{\kappa}{m(m+1)^{\kappa-1}(\kappa-1)}\int_{y=1}^{\frac{x}{A(m+1)}}\frac{x}{\beta}\left(1-\frac{1}{y}\right)^{\frac{x}{\beta}-1}l(y)y^{-\kappa-1}\d y\\
		=&o(x^{-(\kappa-1)}l(x))+\dfrac{\kappa}{m(m+1)^{\kappa-1}(\kappa-1)}\int^{1}_{y=\frac{A(m+1)}{x}}\frac{x}{\beta}(1-y)^{\frac{x}{\beta}-1}l\left(\frac{1}{y}\right)y^{\kappa-1}\d y\\
		=&o(x^{-(\kappa-1)}l(x))+\left(\frac{x}{\beta}\right)^{-(\kappa-1)}l(x)\dfrac{\kappa}{m(m+1)^{\kappa-1}(\kappa-1)}\int_{y=\frac{A(m+1)}{\beta}}^{\infty}e^{-y}y^{\kappa-1} \d y,
	\end{align*}
where we use the dominated convergence theorem in the last equality again. It is also less than ${\epsilon}\tilde{P}_\beta(V>x)/{2\beta} $ when $A$ is large enough. Hence,
 \begin{align*}
     \tilde{P}_\beta(V>x,V\leq AN)&=\beta\tilde{P}_1\left(\beta V> AN,N\geq \frac{x}{A}\right)+\beta\tilde{P}_1\left(\beta V> x,N< \frac{x}{A}\right)\\
     &\le \frac{\epsilon}{2} \tilde{P}_\beta(V>x) +\frac{\epsilon}{2} \tilde{P}_\beta(V>x)= \tilde{P}_\beta(V>x)
 \end{align*}
which completes the claim.
	
	Then we choose some $K\ge{(A+A^2)}/{\epsilon^2}$. Set $V_K=\sum_{i=0}^{K-1} i N_i$. We first prove that $V_K$ stochastically dominates some random variable with regular variation. Indeed, 
	\begin{align*}
	\tilde{P}_\beta(V_K>x)&\geq \tilde{E}_\beta\left[\textbf{1}_{\{V_K>x, V_K>(1-\epsilon)V\}}\textbf{1}_{\{V\leq AN\}}\right]\\
	&\geq \tilde{E}_\beta\left[\textbf{1}_{\{ V_K>(1-\epsilon)V\}}\textbf{1}_{\{V>\frac{x}{1-\epsilon},V\leq AN\}} \right]\\
	&=\tilde{E}_\beta\left[\tilde{P}_\beta\left[V_K>(1-\epsilon)V|V,N;V\leq AN\right]\textbf{1}_{\{V>\frac{x}{1-\epsilon},V\leq AN\}} \right]\\
	&\geq (1-\epsilon)\tilde{P}_\beta\left(V>\frac{x}{1-\epsilon},V\leq AN\right)\\
 	&\geq (1-\epsilon)^2\tilde{P}_\beta\left(V>\frac{x}{1-\epsilon}\right)\\
	&\sim (1-\epsilon)^{1+\kappa}\frac{C^\prime_\kappa \beta }{m(m+1)^{\kappa-1}} x^{-(\kappa-1)}l(x).
	\end{align*}
	where we use Lemma \ref{Lm_Klarge} in the fourth line,  (\ref{eq_VAN}) in the fifth line and Lemma \ref{Lm_Sumofbeta} in the last line. Thus we can find a positive random variable $V^\prime_K$ (without loss of generality, in the same sample space) with the regularly varying tail
\begin{align*}
   \tilde{P}_\beta(V^\prime_K>x)\sim (1-\epsilon)^{1+\kappa}C^\prime_{\kappa}\beta m^{-1}(m+1)^{-(\kappa-1)}   x^{-(\kappa-1)}l(x)
\end{align*}
and 
\begin{align*}
    \tilde{P}_\beta(V_K >x)\geq \tilde{P}_\beta(V^\prime_K>x).
\end{align*}
Thus $\tilde{E}_\beta[(1-s)^{V_K}]\leq\tilde{E}_\beta[(1-s)^{V^\prime_K}]$ for $s\in (0,1)$.
	
Recall that $E_\beta[L^1]=\beta$ and $L^1_{u_j}$ are i.i.d random variables whose distribution only depends on $\beta(u_j)$. Let us consider the Laplace transform of $L^1_0$. By classifying the children by type, we have
	\begin{align*}
	\tilde{E}_\beta\left[e^{-sL_0^1}\right]=\tilde{E}_\beta\left[e^{-s\sum_{j=2}^{N+1}L^1_{u_j}}\right]=\tilde{E}_\beta\left[\prod_{i=0}^{\infty}(\psi_i(s))^{N_i}\right]\leq \tilde{E}_\beta\left[\prod_{i=0}^{K-1}(\psi_i(s))^{N_i}  \right],	
	\end{align*}
	where $\psi_i(s):=E_i[e^{-sL^1}]\leq 1$ when $i\not=1$, $\psi_1(s)=e^{-s}$. Since $\psi_i$ has derivative $-i$ at $s=0$, we get 
	\begin{align*}
	&\tilde{E}_\beta\left[\prod_{i=0}^{K-1}(\psi_i(s))^{N_i} \right]=\tilde{E}_\beta\left[\prod_{i=0}^{K-1}(e^{-s}+o(s))^{i N_i} \right]\\
 =&\tilde{E}_\beta\left[ (e^{-s}+o(s))^{V_K} \right]
	\leq \tilde{E}_\beta \left[(e^{-s}+o(s))^{V^\prime_K}\right]\\
 = &1-  (1-\epsilon)^{1+\kappa} \frac{C^\prime_\kappa \beta }{c_{\kappa}m(m+1)^{\kappa-1}}s^{\kappa-1}l\left(\frac{1}{s}\right)+o\left(s^{\kappa-1}l\left(\frac{1}{s}\right)\right).
	\end{align*}
	Therefore,
	\begin{equation}\label{eq_lowerbound}
	\liminf\limits_{s\rightarrow 0+}\frac{c_{\kappa}m(m+1)^{\kappa-1} s^{-(\kappa-1)}}{l\left(\frac{1}{x}\right)\beta C_\kappa^\prime}\left(1-\tilde{E}_\beta\left[e^{-sL_0^1}\right]\right)\geq  (1-\epsilon)^{1+\kappa}.
	\end{equation}
	Since $\epsilon$ is arbitrary, we get the lower bound.

\textit{2.Upper limit}

We use the Jensen inequality to obtain the upper limit. 
\begin{align*}
    	 &\tilde{E}_\beta\left[e^{-sL_0^1}\right]
    =\tilde{E}_\beta\left[\prod_{j=2}^{N+1}\tilde{E}_\beta\left[e^{-sL^1_{u_j}}\bigg|\beta(u_j)\right]\right]\\
    \geq&\tilde{E}_\beta\left[\prod_{j=2}^{N+1}e^{-s\tilde{E}_\beta\left[L^1_{u_j}|\beta(u_j)\right]}\right]
     =\tilde{E}_\beta\left[e^{-sV}\right]\\
     =&1- \frac{\beta C^\prime_\kappa}{c_\kappa m(m+1)^{\kappa-1}} s^{\kappa-1}l\left(\frac{1}{s}\right)+o\left(s^{\kappa-1}l\left(\frac{1}{s}\right)\right).
\end{align*}
	
	Therefore, 
	\begin{equation}\label{eq_upperbound}
	\limsup\limits_{s\rightarrow 0+}\frac{c_{\kappa}m(m+1)^{\kappa-1} s^{-(\kappa-1)}}{l\left(\frac{1}{x}\right)\beta C_\kappa^\prime}\left(1-\tilde{E}_\beta\left[e^{-sL_0^1}\right]\right)\leq 1.
	\end{equation}
 From (\ref{eq_lowerbound}) and (\ref{eq_upperbound}), we finish the proof.
\end{proof}

We next show that ${x^{\kappa}}\tilde{P}_\beta(L^1_0>x)/{l(x)}$ is bounded by $C\beta^{\kappa}$ where $C$ is a constant that does not depend on $\beta$. We need the additive martingale as a bridge and prepare two lemmas first.

\begin{lemma}\label{Lemma_W0}\cite[Theorem 2.1]{liu2000generalized}
	Let $W_\infty$ to be the a.s. limit of the positive martingale $(W_k)_{k\geq 1}$, where 
 $$(W_k)_{k\geq 1}:=\left(\sum_{|u|=k}\frac{1}{m^k}\right)_{k\geq 1}.$$ 
 Then $W_\infty$ has a finite moment of order $1+\alpha$ under $\mathbb{P}$ for any $\alpha\in(0,\kappa-1)$.
\end{lemma}
Recall that we denote by $P_n$ (resp. $E_n$) the measure (resp. expectation) of the multi-type Galton-Watson tree conditioned on the root having type $n$.
\begin{lemma}\label{Lemma_WL}
	Let  $\alpha\in(0,\kappa-1)$, we have
	\begin{equation*}
		\lim\limits_{n\rightarrow\infty}E_n\left[\left|\frac{L^1}{n}-W_\infty\right|^{1+\alpha}\right]=0.
	\end{equation*}
Moreover, there exists a positive random variable $Y$ with a finite moment of order $1+\alpha$ and $a_n\rightarrow 0$, such that $a_n Y$ is stochastically greater than $|{L^1}/{n}-W_\infty|$ under $P_n$ for any $n\geq 1$.
\end{lemma}
\begin{proof}
	We combine \cite[ Proposition 6]{de2022scaling} and \cite[ Lemma 27]{de2022scaling} to get the lemma.
\end{proof}

We define $W^{u_j}_\infty$ as the almost sure limit of the positive martingale $(W^{u_j}_k)_{k\geq 1}:=(\sum_{|u|=k+1,u\in T^{u_j} }1/m^k )_{k\geq 1} $ .
\begin{proposition}\label{Prop_uni-tail}
	For $L^1_0$ defined as above, it holds with a constant $C$ that 
	\begin{align*}
		\dfrac{x^{\kappa-1}}{l(x)}\hat{P}_\beta(L^1_0>x)\leq C\beta^{\kappa}.
	\end{align*}

\end{proposition}
\begin{proof}
By the Lemma \ref{Lemma_WL}, we can find i.i.d random variables $Y_{u_j},2\leq j\leq N+1$ with finite $1+\alpha$ moments ($\alpha\in (0,\kappa-1)$) such that ${L^1_{u_j}}/{\beta(u_j)}-W^{u_j}_\infty$ is stochastically dominated by $Y_{u_j}$ under $P_{\beta(u_j)}$. We define $Z_{u_j}:=Y_{u_j}+W^{u_j}_\infty$, then $Z_{u_j}$ are i.i.d random variables with finite $1+\alpha$ moment by Lemma \ref{Lemma_WL} and $L^1_{u_j}$ is stochastically dominated by $\beta(u_j)Z_{u_j}$. Note that the law of $Z_{u_j}$ does not depend on $\beta(u_j)$.

Now let $\beta_k(u_j)$ be the edge local time of $u_j$ in the $k$-th trail. We see that $(\beta(u))_{u\in \Omega(\omega_1)}$ are i.i.d. Thus

\begin{align*}
	\tilde{P}_\beta(L^1_0>x)\leq &\tilde{P}_\beta\left(\sum_{j=2}^{N+1}\beta(u_j)Z_{u_j}>x\right)\\
 \leq &\sum_{k=1}^{\beta}\tilde{P}_1\left(\sum_{j=2}^{N+1}\beta_k(u_j)Z_{u_j}>\frac{x}{\beta}\right) \\=&\beta\tilde{P}_1\left(\sum_{j=2}^{N+1}\beta(u_j)Z_{u_j}>\frac{x}{\beta}\right) .
\end{align*}

Let $m_Z:=E[Z]<\infty$. By the Laplace transform, following the deduction in Proposition \ref{Prop_main}, we see
\begin{align*}
	\tilde{E}_1\left[e^{-s\sum_{j=2}^{N+1}\beta(u_j)Z_{u_j}}\right]\sim\dfrac{C^\prime_\kappa m_Z^{\kappa-1}}{c_\kappa m(m+1)^{\kappa-1}}s^{\kappa-1}l\left(\frac{1}{s}\right).
\end{align*}  
Thus $\sum_{j=2}^{N+1}\beta(u_j)Z_{u_j}$ is also regularly varying and we can find a constant $C$ independent of $\beta$ such that
\begin{align*}
	\beta\tilde{P}_1\left(\sum_{j=2}^{N+1}\beta(u_j)Z_{u_j}>\frac{x}{\beta}\right)\leq C\beta^{\kappa}x^{-(\kappa-1)}l(x).
\end{align*}
\end{proof}

Now, we are ready to prove the main proposition and finish the proof of Theorem \ref{Thm_mian}. 

\begin{proof}[proof of Proposition \ref{Prop_rvofL1}]
    
By Proposition \ref{Prop_main},
\begin{align*}
    \hat{P}_1(L^1_{\Omega(\omega_{k})}>x|\beta(\omega_k),\beta(\omega_{k-1}))\sim \frac{C^\prime_\kappa (\beta(\omega_{k-1})+\beta(\omega_k)) }{m(m+1)^{\kappa-1}}l(x)x^{-(\kappa-1)}. 
\end{align*}
From the spinal decomposition,
\begin{align*}
\lim\limits_{x\rightarrow\infty}\frac{x^{\kappa-1}}{l(x)}\hat{P}_1(L^1>x)=\lim\limits_{x\rightarrow\infty}\frac{x^{\kappa-1}}{l(x)}\hat{E}_1\left[\hat{P}_1\left(\left(\sum_{k=1}^{\hat{\tau}_1}L^1_{\Omega(\omega_{k})}+1\right)>x\Bigg|(\beta(\omega_{k}))_{k}\right)\right].
\end{align*}
Almost surely $\hat{\tau}_1$ is finite and $L^1_{\Omega(\omega_{k})}$ are independent of each other conditioned on $(\beta(\omega_{k}))_{k}$, with regular variation of the same order. By considering the Laplace transform \cite[Theorem 8.1.6]{bingham1989regular}, conditioned on $(\beta(\omega_k))_k$, $\sum_{k=1}^{\hat{\tau}_1}L^1_{\Omega(\omega_{k})}$ is regular varying, that is
\begin{align*}    \hat{P}_1\left(\left(\sum_{k=1}^{\hat{\tau}_1}L^1_{\Omega(\omega_{k})}+1\right)>x\Bigg|(\beta(\omega_{k}))_{k}\right)\sim  \sum_{k=1}^{\hat{\tau}_1} \frac{C^\prime_\kappa (\beta(\omega_{k-1})+\beta(\omega_k)) }{m(m+1)^{\kappa-1}}l(x)x^{-(\kappa-1)}.
\end{align*}
Moreover, we have
\begin{equation*}
\hat{P}_1\left(\sum_{k=1}^{\hat{\tau}_1}L^1_{\Omega(\omega_{k})}>x\Bigg|(\beta(\omega_{k}))_{k}\right)\leq \sum_{k=1}^{\hat{\tau}_1} \hat{P}_{\beta(\omega_{k-1})}\left(L^1_{\Omega(\omega_{k})}>\frac{1}{2k^2}x\right).
\end{equation*}
It implies that 
\begin{align*}
    \frac{x^{\kappa-1}}{l(x)}\hat{P}_1\left(\sum_{k=1}^{\hat{\tau}_1}L^1_{\Omega(\omega_{k})}>x\Bigg|(\beta(\omega_{k}))_{k}\right)\le
C\left(\sum_{k=1}^{\hat{\tau}_1+1} \beta(\omega_{k-1})^{\kappa}k^{2(\kappa-1)+\epsilon}l(k^2)\right),
\end{align*}
where $C$ is a constant by Proposition \ref{Prop_uni-tail}. It is integrable by Lemma \ref{Lm_sumbeta}.

Therefore, by the dominated convergence theorem,
\begin{align*}                                                       
\lim\limits_{x\rightarrow\infty}\frac{x^{\kappa-1}}{l(x)}\hat{P}_1(L^1>x)&=\hat{E}\left[	\lim\limits_{x\rightarrow\infty}\frac{x^{\kappa-1}}{l(x)}\hat{P}_1\left(\left(\sum_{k=1}^{\hat{\tau}_1}L^1_{\Omega(\omega_{k})}+1\right)>x\Bigg|(\beta(\omega_{k}))_{k}\right)\right]\\
&=\frac{C^\prime_\kappa}{m(m+1)^{\kappa-1}}\hat{E}_1\left[\sum_{k=1}^{\hat{\tau}_1} \beta(\omega_{k-1})+\beta(\omega_{k})\right]\\
&=\frac{C^\prime_\kappa}{m(m+1)^{\kappa-1}} 2 \hat{E}_1\left[\sum_{k=1}^{\hat{\tau}_1} \beta(\omega_{k})\right].
\end{align*}
Since $(\beta(\omega_k))_{k\geq 0}$ is a Markov chain with transition probability $\hat{P}(\beta(\omega_k)=j| \beta(\omega_{i-1})=i)=(b_j/b_i)m_{ij}$,  the stationary measure associated with it is $(a_ib_i)_{i\geq 1}=(m^{-i-1}i(m-1)^2)_{i\geq 1}$. By the theory of Markov chains, it follows that
\begin{align*}
    \hat{E}_1\left[\sum_{k=1}^{\hat{\tau}_1} \beta(\omega_{k})\right]=\sum_{i\geq 1} \dfrac{i a_ib_i }{a_1b_1}=\frac{m^2(m+1)}{(m-1)^3}.
\end{align*}

In summary, we get that, as $x\rightarrow\infty$
\begin{equation*}
\hat{P}_1(L^1>x)\sim C_\kappa x^{-(\kappa-1)}l(x),
\end{equation*} 
where 
\begin{align*}
    C_\kappa=\frac{2\Gamma(\kappa)\kappa m}{(m-1)^3(m+1)^{\kappa-2}(\kappa-1)}.
\end{align*}
\end{proof}




\bibliographystyle{amsplain}
\bibliography{scaling}
\nocite{*}
\end{document}